\documentclass[11pt]{amsart}
\usepackage{amsmath,amssymb,amscd,fullpage,pb-diagram,hyperref}

\theoremstyle{definition}

\newcommand{\bd}{\left|\begin{matrix}}
\newcommand{\ed}{\end{matrix}\right|}
\newcommand{\tr}{{}\;^\top\!\!}
\newtheorem{thm}[equation]{Theorem}
\newtheorem{cor}[equation]{Corollary}

\newtheorem{prop}[equation]{Proposition}
\newtheorem{lem}[equation]{Lemma}
\newtheorem{rmk}[equation]{Remark}

\newtheorem{defn}[equation]{Definition}

\numberwithin{equation}{section}

\newcommand{\bpm}{\begin{pmatrix}}
\newcommand{\epm}{\end{pmatrix}}
\newcommand{\bsm}{\begin{smallmatrix}}
\newcommand{\esm}{\end{smallmatrix}}
\newcommand{\bspm}{\left(\begin{smallmatrix}}
\newcommand{\espm}{\end{smallmatrix}\right)}
\newcommand{\bm}{\begin{matrix}}
\renewcommand{\em}{\end{matrix}}
\newcommand{\bbm}{\begin{bmatrix}}
\newcommand{\ebm}{\end{bmatrix}}

\newcommand{\C}{\mathbb{C}}

\newcommand{\Z}{\mathbb{Z}}
\newcommand{\D}{\mathbb{D}}
\newcommand{\R}{\mathbb{R}}
\newcommand{\N}{\mathbb{N}}
\newcommand{\M}{\mathbb{M}}
\newcommand{\Tr}{\operatorname{Tr}}

\newcommand{\Ad}{\operatorname{Ad}}

\newcommand{\pr}{\operatorname{pr}}

\newcommand{\Mer}{\operatorname{Mer}}

\newcommand{\sym}{\operatorname{sym}}

\newcommand{\Span}{\operatorname{Span}}

\newcommand{\tx}{\text}

\newcommand{\ptl}[2]{\frac{\partial #1}{\partial #2}}
\newcommand{\dd}[2]{\frac{d #1}{d #2}}
\newcommand{\ddt}{\dd{}t}
\newcommand{\ve}{\varepsilon}

\newcommand{\on}{\operatorname}
\newcommand{\ol}{\overline}
\newcommand{\ul}{\underline}

\newcommand{\g}{\mathfrak{g}}
\newcommand{\fg}{\mathfrak{g}}
\newcommand{\fk}{\mathfrak{k}}
\newcommand{\fp}{\mathfrak{p}}
\newcommand{\fr}{\mathfrak{r}}
\newcommand{\fu}{\mathfrak{u}}
\newcommand{\fh}{\mathfrak{h}}
\newcommand{\f}{\mathfrak}
\newcommand{\fsl}{\mathfrak{sl}}
\newcommand{\fgl}{\mathfrak{gl}}

\newcommand{\wt}{\widetilde}
\newcommand{\Mat}{\on{Mat}}
\newcommand{\ad}{\on{ad}}
\newcommand{\cJ}{\mathcal{J}}
\renewcommand{\c}{\mathcal}
\newcommand{\gm}{\gamma}

\newcommand{\Sp}{\wt{Sp}}
\newcommand{\Gm}{\Gamma}
\newcommand{\sg}{\sigma}
\newcommand{\Ja}{G^{(n,j)}}
\newcommand{\jla}{\fg^{(n,j)}}
\newcommand{\JC}{\wt{G}^{(n,j)}}
\renewcommand{\t}{\,^\top \!}

\begin{document}
\author{Joseph Hundley} \email{jhundley@math.siu.edu}
\address{Math. Department, Mailcode 4408,
Southern Illinois University Carbondale, 1245 Lincoln
Drive Carbondale, IL 62901}
\thanks{Most of this research was completed in 2006-07 while the author was a post-doc at Pohang University of Science and Technology and then at Tel Aviv University.  He would like to thank both institutions for their hospitality, as well as his hosts, Professor YoungJu Choie and Professor David Ginzburg, respectively.  He would also like to thank Professor Choie for suggesting the problem and for many useful discussions.
This paper was written while the author was supported
by NSF Grant DMS-1001792.}
\title{On some results of Bump-Choie and Choie-Kim}
\date{\today}
\maketitle
\tableofcontents
\section{Introduction}
The purpose of this note is to give 
a representation-theoretic interpretation for the results 
of Choie and Kim \cite{Choie-Kim}, continuing a program 
 begun by Bump and Choie in \cite{Bump-Choie}.
The motivation for \cite{Choie-Kim} is to generalize a certain result
of Bol \cite{Bol} from elliptic modular forms to Siegel and Jacobi forms.
An interesting consequence is that the $(r+1)^{\text{\ul{st}}}$ derivative
of a meromorphic modular form of weight $-r$ is a meromorphic
modular form of weight $r+2.$  This is striking because no other 
derivative of the same modular form will be modular.  More
precisely, each derivative is expressed as a linear combinations
of forms of various weights, only the $(r+1)^{\text{\ul{st}}}$ derivative
has the property that the contributions from all weights save one 
are zero.  As explained in \cite{Bump-Choie}, it is possible to 
consider instead a slightly different differential operator 
(the Maass raising operator) such that all derivatives 
are automorphic forms, but none save the $(r+1)^{\text{\ul{st}}}$ 
is meromorphic.

The results of \cite{Choie-Kim}
extend Bol's result to Siegel modular forms as well as Jacobi 
forms, subject to certain restrictions.  
For the Siegel case, 
one considers a meromorphic Siegel modular form 
of genus $n$ and weight $-r+\frac{n-1}2.$ 
The usual derivative
$\frac{d}{dz}$ of a modular form in the upper half plane is 
replaced by  the ``determinant differential 
operator'' in the Siegel case. 
Just as before, one finds that  the $(r+1)^{\text{\ul{st}}}$
iterate of this operator, yields a 
meromorphic Siegel or Jacobi form of weight $r+2+\frac{n-1}2.$  Further
one does not get a meromorphic modular form for any 
derivative other than  the $(r+1)^{\text{\ul{st}}}.$
The Jacobi case is similar.
Let us refer to this as ``recovery at $r+1.$''

The Siegel case is considered from another perspective in \cite{Bump-Choie}, and 
the important task of translating the phenomenon 
into a representation-theoretic 
context is carried out. 
It is explained how  certain vectors in $(\mathfrak{sp}(2n,\C), \wt U(n))$-modules generalize meromorphic modular forms of negative weight, 
and how a conjecture (Conjecture 2 of \cite{Bump-Choie})
would explain and also generalize the natural extension (Conjecture 1 of 
\cite{Bump-Choie}) of recovery at $r+1$.

The goals of this paper are as follows:  first, to prove 
Conjecture 2 of \cite{Bump-Choie}; second, to generalize Theorem 3 of 
\cite{Bump-Choie}; third, to give a similar treatment of the Jacobi 
case of \cite{Choie-Kim}.

Two proofs of Conjecture 2 of \cite{Bump-Choie}
are given.  The first, short proof, is based on 
deducing the conjecture from general results of 
  \cite{Boe}, which 
apply to any 
semisimple complex 
Lie algebra.
Since the main program is to make the representation-theoretic
explanations for the ``recovery'' behavior of Siegel and Jacobi
modular forms as transparent as possible, a detailed proof for the 
case $\mathfrak{sp}(2n, \C)$ is included. 
A second, longer, proof, along the lines of that given in 
\cite{Bump-Choie} for the case $n=2$ is then presented.  

The generalization of theorem 3 of \cite{Bump-Choie} is given in section \ref{section: other elements of the center}.  It helps to explain why, regardless of the the genus $n,$
one is able to prove Conjecture 2 using only the Laplace-Beltrami operator, without considering any other elements of the center of the universal enveloping algebra.

In section \ref{section: jacobi case notation}, notation 
for the Jacobi case is introduced.  The main theorem
in the Jacobi case, theorem \ref{Main thm: jacobi}, is then given in section \ref{section: Main result a la Choie-Kim}.  It generalizes corollary 3.6 of \cite{Choie-Kim}.  The proof is based on extending Jacobi forms to functions on a large Siegel upper half-space and applying results from the Siegel case.  In section \ref{section:  Alternate approach, a la Bump-Choie}, theorem \ref{Jacobi U(g) version of Conjecture 2}, which may be regarded as a Jacobi version of Conjecture 2 is stated and proved.  In principle, 
this result could be the basis for a second proof of theorem \ref{Main thm: jacobi}, but 
it's not entirely trivial: like Conjecture 2
of \cite{Bump-Choie}, theorem \ref{Jacobi U(g) version of Conjecture 2} deals with certain elements of the 
universal enveloping algebra which correspond to 
Maass's raising operators.  The relationship between 
these raising operators and the heat operator and 
determinant differential operator considered in \cite{Choie-Kim} is somewhat indirect.  
In the Siegel case, the relationship is given in 
lemma \ref{M is basically D}, which relies on 
results of Harris and Maass.  
In order to deduce theorem \ref{Main thm: jacobi}
from theorem \ref{Jacobi U(g) version of Conjecture 2}, 
some Jacobi
version of this lemma would be required.
Even though no second proof of theorem \ref{Main thm: jacobi} is obtained, it seems that theorem \ref{Jacobi U(g) version of Conjecture 2}, and particularly its proof, may be of sufficient independent interest to warrant inclusion.  
The proof is based on embedding a smaller symplectic
Lie algebra in the universal enveloping algebra of the Jacobi Lie algebra in a nonstandard way, and 
applying results from the Siegel case to this 
subalgebra.  

\section{Reduction}\label{s: reduction}
As mentioned above, this paper attempts to continue a program 
begun in \cite{Bump-Choie}.  
Accordingly, we make free use of notation and  terminology 
introduced in  \cite{Bump-Choie}, and generally begin not
at the beginning, but where \cite{Bump-Choie} leaves off.

To briefly recall the statement of Conjecture 2, recall that
a $(\g_\C, \wt K)$-module is a vector space equipped with 
actions of $\g_\C:= \frak{sp}(2n, \C)$  and of $\wt K=\wt U(n),$
subject to certain compatibility conditions described on 
\cite{Bump-Choie}, pp. 116-17.
Here $\wt U(n)$
 is the preimage of the unitary group $U(n) \subset Sp(2n, \R)$ 
in the metaplectic double cover $\wt {Sp}(2n, \R)$ of $Sp(2n, \R).$ 

Let $(\pi, V)$ be such an object.  A vector  $v \in V$ is said to be semispherical
if the action of $\wt K$ preserves the one dimensional subspace $v$ 
spans.
In this case $\wt K$ acts on $v$ by a character.
Following \cite{Bump-Choie}, p. 121, we denote this 
character by $\det^{k},$ with $k$ being a half integer.
The vector $v$ is said to be holomorphic if it is semispherical and, 
in addition, annihilated by certain differential operators $\pi(L_X)$ 
described on \cite{Bump-Choie}, p. 120. 
Rather than the determinant differential operator,
one works with a different, but closely 
related differential operator, 
namely the Maass raising operator 
 $M_+$  defined on p. 123 of \cite{Bump-Choie}.
 Conjecture 2 of \cite{Bump-Choie} then states that if 
 $v$ is a holomorphic vector of weight $-r+\frac{n-1}2,$ then $M_+^{r+1} v$
 is holomorphic of weight $r+2+\frac{n-1}2.$

In order to prove this, it is useful to simply forget the $\wt K$-module
structure, and regard $V$ as a module over the ring $U(\g_\C).$  
Let $\fk$ denote the Lie algebra of $\wt K.$  
Then the compatibility 
conditions of a $(\g_\C, \wt K)$-module force $\fk$ to act on the 
span of $v$ by a character, $\pi(X) v = k \Tr(X) v,$ where 
$k \in \frac12 \Z$ is the weight, and $\Tr(X)$   denotes the 
trace of $X \in \fk,$ regarded as an element of $\mathfrak{gl}(n, \C),$
as opposed to $\frak{sp}(2n, \R).$ 
To be precise, 
let $c$ be the Cayley transform, as on  \cite{Bump-Choie}, p. 120.
Then for  $X \in \fk,$ the matrix $c X c^{-1}$ is of the form $\bspm X_0 & \\ & {\ol X}_0
\espm,$ with $X_0 \in \mathfrak{su}(n)$ and $\Tr(X)$ denotes
the trace of $X_0.$  
Note that $\ol X_0 = -\,^\top\!\! X_0,$ for $X_0 \in \mathfrak{su}(n).$
It follows by linearity that the vector $v_0 := \pi(c) v$ satisfies 
\begin{equation}\label{wt(k)} \tag{Weight $k$}
\pi\bpm X & \\ & -\tr X \epm v_0 = k \Tr(X) v_0 \qquad \forall X \in \mathfrak{gl}(n, \C).
\end{equation}
Moreover, by holomorphicity, (cf. the definition of $L_X$ on \cite{Bump-Choie}, p. 120)
\begin{equation}\label{hol}\tag{Hol}
\pi\bpm 0 &0 \\X & 0 \epm v_0 = 0,
\qquad \forall \,X \in \Mat_n(\C),\text{ symmetric.}
\end{equation}
Moreover, if $V$ is a $(\g_\C, \wt K)$-module, and $v_0 \in V$ satisfies 
\eqref{wt(k)}, then it follows from the compatibility 
conditions for a $(\g_\C, \wt K)$-module that $v_0$ 
is semispherical of weight $k.$
Hence, Conjecture 2 of \cite{Bump-Choie} reduces to the 
following proposition.
\begin{prop}\label{U(g) version of Conjecture 2}
Let $V$ be a $U(\g_\C)$ module, and suppose that 
$V$ contains a 
vector $v_0$ satisfying \eqref{hol} and 
(\hyperref[wt(k)]{Weight $-r+ \frac{n-1}2$}).  
Then $\hat M_+^{r+1} v_0$ satisfies \eqref{hol} and 
(\hyperref[wt(k)]{Weight $r+2+ \frac{n-1}2$}).
Moreover, if $k \ne r+1,$ then $\hat M_+^k v_0$ 
does not satisfy \eqref{hol}, 
where $\hat M_+ = \det \hat R_{X_{ij}},$ with the same 
convention as on p. 123 of \cite{Bump-Choie}.
\end{prop}

\section{Short proof of proposition \ref{U(g) version of Conjecture 2}}
One may give a short proof of proposition \ref{U(g) version of Conjecture 2}
using a modest refinement of theorem 4.4 of \cite{Boe}.  Thus, 
take $\g$ now to be a finite dimensional complex semisimple Lie algebra
as in \cite{Boe}, and let $\fp$ be a hermitian symmetric parabolic subalgebra.  Write $\fp = \fu^+ \oplus \fr$ where $\fu^+$ is the unipotent 
radical of $\fp$ and $\fr$ is reductive, and write $\fr = \g_S + \fh,$ 
where $\g_S$ is the derived subalgebra and $\fh$ is a Cartan subalgebra.
Let  $\lambda: \fh \to \C$ be a  linear mapping  which is trivial on 
 $\fh \cap \g_S,$ 
 and let $\pr$ be the natural projection $\fp \to \fh/(\fh \cap \g_S).$
 One considers those vectors $v$ in a $U(\g)$-module $V$
 which satisfy the following 
  natural generalization of \eqref{hol} and
 \eqref{wt(k)}: \begin{equation}\label{hol wt lambda}
 \tag{Hol wt $\lambda$}
 X v = \lambda( \pr(X)) v\qquad( \forall X \in \fp).
 \end{equation}
 Then one has, from \cite{Boe}, a necessary and sufficient 
 condition for a representation which is generated by a vector 
 which satisfies \eqref{hol wt lambda} to contain a vector which 
 satisfies (\hyperref[Hol wt lambda]{Hol wt $\lambda'$}), 
 where $\lambda' \ne \lambda.$  In order to state this condition, 
 one needs a bit more notation.  First, let $\fu^-$ denote the nilpotent radical of the parabolic subalgebra
 opposite $\fp,$ so that $\fg = \fp \oplus \fu^-,$ and
 let 
 $ \rho^S$ equal half the sum of the roots of $\fh$ in $\fu^+.$
 Next, let $\alpha$ denote the unique noncompact simple root of $\fh,$  
let $\fg_\alpha$ be the $\alpha$-eigenspace of $\fh$ in $\fg,$
let  
$ h_\alpha$ denote the unique element of $[\g_\alpha, \g_\alpha]$
with $\alpha(h_\alpha)=2.$ 
Then the space $\fh/(\fh \cap \g_S)$ is one dimensional and 
spanned by the image of $h_\alpha.$  Let $\omega_\alpha:\fh \to \C$
be the unique element which is  trivial on $(\fh \cap \g_S)$
and sends $h_\alpha$ to $1.$
 \begin{thm}\label{boesThm}
 Keep the notation of the previous paragraph and let 
 $\lambda, \lambda': \fh/(\fh \cap \g_S) \to \C$ 
 be two linear maps.
The following are equivalent.
\begin{enumerate}
\item There exists a $U(\g)$-module $V,$ which is generated by a
vector $e$ which satisfies \eqref{hol wt lambda} 
and contains a second vector $e'$ satisfying 
(\hyperref[Hol wt lambda]{Hol wt $\lambda'$}),
\item There is a nontrivial $U(\g)$-module map from the scalar generalized 
Verma module $U(\g)\otimes_{U(\fp)} \C e_{\lambda'}$ to 
the scalar generalized 
Verma module $U(\g)\otimes_{U(\fp)} \C e_{\lambda},$
\item The algebra of $\g_S$-invariants in $U(\fu^-)$ 
is of the form $\C[u_r],$ and there  
 is a nonnegative integer $k$ such that 
$$
\lambda = (k-\rho^S(h_\alpha))\omega_\alpha, \qquad 
\lambda' = (-k-\rho^S(h_\alpha))\omega_\alpha.
$$
\end{enumerate}  
Moreover, when the three equivalent conditions hold, 
the vector $e'$ mentioned in (1) is a scalar multiple of $u_r^k \,.\,e,$ 
and the map mentioned in (2) sends $1\otimes e_{\lambda'}$ to a scalar 
multiple of $u_k^r \otimes e_\lambda.$ (Note that this 
map is completely determined by its value on $1\otimes e_{\lambda'}.$)
\end{thm}
\begin{proof}[Sketch of proof]The statement given above goes
slightly beyond what is stated in \cite{Boe}, but the arguments in 
\cite{Boe} prove it.  Indeed, it is clear that (2) $\implies$ (1), 
for if $\phi: U(\g)\otimes_{U(\fp)} \C e_{\lambda'}
\to U(\g)\otimes_{U(\fp)} \C e_{\lambda}$ is any nonzero map, 
then (1) is satisfied with $V=U(\g)\otimes_{U(\fp)} \C e_{\lambda},
e=e_\lambda$ and $e' = \phi(e_{\lambda'}).$  
Further, if $V$ is any $U(\fg)$-module generated by 
a vector $e$ which satisfies \eqref{hol wt lambda}, then $V = U(\fu^-)e.$
Arguing as on p. 796 of \cite{Boe}, if $u \in U(\fu^-)$ is such that 
 $X\,.\,(u \,.\, e) = \lambda'(\pr(X)), \, (\forall X \in \fr),$ then it follows that $u \in U(\fu^-)^{\fg_S},$
 and $\ad(X) u = ( \lambda' - \lambda)(\pr(X)) \cdot u, \, (\forall X \in \fr).$
   By proposition 4.2 of \cite{Boe}, the space $ U(\fu^-)^{\fg_S}$ 
 is either $\C,$ or else it is $\C[u_r],$ for a certain element $u_r.$
 which satisfies 
 $\ad(X) u_r = -2\omega_\alpha(X) u_r$ (cf. the computation of ``$\mu_r$''
 at the end of the proof of theorem 4.4).
 
 The heart of the proof of \cite{Boe} theorem 4.4 is a proof that 
 $u_r^k \otimes e_{\lambda}$ is annihilated by $\fu^+$ if and only if 
 $\lambda = (k-\rho^S(h_\alpha))\omega_\alpha.$  It goes through 
 word for word if we replace $ u_r^k \otimes e_{\lambda}$ 
 by $u_r^k e$ for any $e$ which satisfies \eqref{hol wt lambda},
 provided $u_r^k e\ne 0.$
 This proves that (1) $\implies$ (3), and that (3) $\implies$ (2).  
\end{proof}
 
 \begin{proof}[Proof of proposition \ref{U(g) version of Conjecture 2}]
 Proposition  \ref{U(g) version of Conjecture 2} is a fairly direct 
 application of theorem \ref{boesThm}.  If $c$ is the 
 Cayley transformation as in \cite{Bump-Choie}, one takes $\fg = \mathfrak{sp}(2n, \C),$ and 
 \begin{equation}\label{definition of frak p}
\fp = \left\{ c^{-1}\bpm X_0 &0\\X& -\,^\top\!\! X_0 \epm c:
X_0 \in \mathfrak{gl}(n, \C), \; X \in \Mat_n(\C), \text{ symmetric}\right\}.
\end{equation}
$$
\fu^- = \left\{c^{-1}\bpm 0&X \\ 0&0 \epm c: X \in \Mat_n(\C), \text{ symmetric} \right\}.
$$
\begin{equation}\label{fr and fgS}
\fr = \left\{c^{-1} \bpm X_0 &0\\0& -\,^\top\!\! X_0 \epm c:
X_0 \in \mathfrak{gl}(n, \C)\right\}, 
\qquad \g_S = \left\{ c^{-1}\bpm X_0 &0\\0& -\,^\top\!\! X_0 \epm c:
X_0 \in \mathfrak{sl}(n, \C)\right\}.
\end{equation}
The root $\alpha$ is the unique long simple root of $\mathfrak{sp}(2n, \C).$
The element $c^{-1}h_\alpha c$ has a $-1$ at $n,n,$ a $1$ at $2n,2n,$ and 
zeros everywhere else, and $\rho^S(c^{-1} h_\alpha c) = \frac{n+1}2.$
  Also
 $$\omega_\alpha \circ \pr\left[c^{-1}\bpm X_0 &0\\X& -\,^\top\!\! X_0 \epm c\right]
 = -\Tr(X_0).$$
 
 What remains is to check that the generator $u_r$ for $U(\fu^-)^{\fg_S}$
 is the element $M_+$ defined above.
 This is not difficult.  Since $\fu^-$ is commutative, 
 the algebra $U(\fu^-)$ may 
 be identified with the symmetric algebra $S(\fu^-).$
We my identify $X\in \fu^-$ with the linear form 
$Y \mapsto \Tr(XY), \fu^+ \to \C,$ and thus identify 
$S( \fu^-)$ with the algebra of polynomial-functions on $\fu^+.$
 The space of $SL_n(\C)$-invariant symmetric polynomials 
 on $\fu^+$ is well-known  to be generated by 
  the map $\bspm 0&0 \\ X&0 \espm\mapsto \det X,$
  and one easily checks that the corresponding element of $S(\fu^-)$ 
  is $M_+.$
\end{proof}
\begin{rmk}
The preceding proof consists mainly of 
explaining what the key players in 
 Boe's general theory are in the special case 
  $\f g=\f{sp}(2n,\C).$
These details are also worked out in detail in Boe's Ph. D. thesis \cite{Boe2}.
\end{rmk}

\section{Longer proof of proposition \ref{U(g) version of Conjecture 2}}
The purpose of this section is to give a proof of proposition \ref{U(g) version of Conjecture 2} which is along the lines of the proof given for the case
$n=2$ in \cite{Bump-Choie}.
Thus, we return to the notation of section \ref{s: reduction}:  $\fg 
= \mathfrak{sp}(2n, \C),$ etc.
\subsection{Step 1:}  The first step is to show that the weight 
of a holomorphic and semispherical vector determines the 
eigenvalue of the Laplace-Beltrami operator acting on
that vector, and hence on the entire $U(\fg)$-module that it
generates.  Recall the Laplace-Beltrami operator is a differential 
operator of degree two which lies in the center of $U(\g),$
and that these conditions determine it uniquely up to scalar.  
As noted on p. 129 of \cite{Bump-Choie}, a theorem of Harish-Chandra
implies that this operator is the image under the symmetrization 
map $\lambda$  (\cite{Bump-Choie}, p. 128) of an $\ad$-invariant
element of $S(\g).$  Such elements of $S(\g)$ may be identified with $\ad$-invariant polynomial-functions on $S(\g)$ using  the 
$\ad$-invariant bilinear form $(X,Y) := \Tr(XY).$  The Laplace-Beltrami 
operator may then be normalized to be the operator $\Delta$
corresponding to the polynomial-function $A \mapsto \Tr(A^2).$

This operator can be written out explicitly in terms of a basis for $\fg.$ 
Write $e_{i,j}$ for the matrix with $1$ at $i,j$ and zeros everywhere else.
Take $d_i = e_{i,i}-e_{n+i,n+i}$ for each $i=1$ to $n.$  Then 
$\{ d_i: 1 \le i \le n\}$ is a basis for the standard Cartan subalgebra 
of $\g.$  Write $\{ d_i^*: 1 \le i \le n\}$ for the dual basis.  Then the positive
roots of $\fh$ in $\g$ are 
$$
\{ d_i^*-d_j^*: 1 \le i < j \le n\} \cup \{ d_i^*+d_j^*: 1 \le i < j \le n\}\cup 
\{2d_i^*: 1 \le i \le n\}.
$$
Roots from one of the first two sets above are short and roots from the last are long.
For each positive root define $X_\alpha$ to be 
$$
X_\alpha := \begin{cases}e_{i,j}-e_{n+j,n+i}, & \alpha = d_i^*-d_j^*,\\
e_{i, n+j}+ e_{j, n+i}, & \alpha = d_i^*+d_j^*,\\
e_i, n+i, & \alpha = 2d_i^*.\end{cases}
$$
In all cases define $Y_\alpha$ to be the transpose.
Then teasing out the definitions, one obtains
$$\begin{aligned}
\Delta &= \lambda\left( 
\sum_{i = 1}^n d_i^2 + 2\sum_{\alpha \text{ short } } 
X_\alpha Y_\alpha+4\sum_{\alpha \text{ long }} 
X_\alpha Y_\alpha\right)\\
&=\sum_{i = 1}^n d_i^2 + \sum_{\alpha \text{ short } } 
(X_\alpha Y_\alpha+ Y_\alpha X_\alpha) + 2\sum_{\alpha \text{ long }} 
(X_\alpha Y_\alpha+ Y_\alpha X_\alpha)\\
&=\left(\sum_{i=1}^n d_i^2 -\sum_{\alpha} c_\alpha [X_\alpha,Y_\alpha]\right)
+\left( 2 \sum_\alpha c_\alpha X_\alpha Y_\alpha\right),
\end{aligned}
$$
where  the sums are over positive roots $\alpha$ and  $c_\alpha$ is $2$ if $\alpha$ is long and $1$ if $\alpha$ is short.
\begin{lem}
Let $V$ be a $U(\g)$-module generated by  a vector $v$ which satisfies
\eqref{hol} and \eqref{wt(k)}.  Let $\Delta$ denote the Laplace-Beltrami
operator, let 
$$
\Delta^{(1)} = \left(\sum_{i=1}^n d_i^2 -\sum_{\alpha} c_\alpha [X_\alpha,Y_\alpha]\right),
$$
and let $\hat M_+=cM_+c^{-1},$ where $M_+$ is the operator defined on p. 123 of \cite{Bump-Choie}.  
Equivalently, $\hat M_+ = \det \hat R_{X_{ij}},$ with the same 
convention as on p. 123 of \cite{Bump-Choie}.
Then   
\begin{equation}\label{action of del1}
 \Delta^{(1)} M_+^r v = (k+2r)n(k+2r-n-1) M_+^r v, \qquad (r=0,1,2,3,\dots)
 \end{equation}
\begin{equation}\label{action of del}
 \Delta w= kn(k-n-1) w, \qquad (\forall w \in V).
 \end{equation}
\end{lem}
\begin{proof}
One first checks \eqref{action of del1} using \eqref{hol}, together
with Proposition 5 of p.124 of \cite{Bump-Choie}. 
Equation \eqref{action of del}
for $w=v$ then follows since $\Delta-\Delta^{(1)}$ kills any holomorphic vector.   Equation \eqref{action of del} 
then holds for all $w$ because $\Delta$ is in the center of $U(\g).$
\end{proof}
\begin{rmk}
This also gives a second proof that the $U(\g)$-module generated
by a holomorphic vector can contain at most one other linearly independent
holomorphic vector, and that such a vector may exist 
only when the weight of the generator is of the form $-r+\frac{n-1}2,$
in which case the second linearly independent vector must be of weight
$r+2+ \frac{n-1}2.$
\end{rmk}
\begin{cor}
If $v$ and $r$ are as above, and furthermore $k=-r+\frac{n-1}{2}$ 
then $$\left( 2 \sum_\alpha c_\alpha X_\alpha Y_\alpha\right) \hat M^{r+1}_+ v = 0.$$  Here the sum may be taken over all positive 
roots, or only over the roots of $\fh$ in $\f R,$ 
(defined on p. 122 of \cite{Bump-Choie})
since $\hat M^{r+1}_+ v$
satisfies 
(\hyperref[wt(k)]{Weight $r+2+\frac{n-1}2$}). 
\end{cor}
\begin{proof}
If $r,k,$ and $n$ stand in this relationship to one another then 
$kn(k-n-1)=(k+2r)n(k+2r-n-1)$ and our claim follows by subtracting the two 
equations of the previous corollary.
\end{proof}
\subsection{An isomorphism}  The next step to mimic the argument 
given at the bottom of p. 131 of \cite{Bump-Choie}.  To that end 
define $u=\hat M^{r}_+v_0$ and $w=\hat M^{r+1}_+ v_0,$
where $c$ is the Cayley transform.
Let $\f L = \Span \left\{ \bspm 0&0\\X&0\espm\mid 
X \in \Mat_n\C,\tx{ symmetric}\right\}.$
The goal is to prove that $w$ is annihilated by 
$U(\f L) \subset U(\f g).$ 
  If $w$  is zero, then there is nothing to prove.  Therefore, one 
may assume that $w$ is nonzero, in which case clearly $u$ is also 
nonzero.

Now, recall that $\hat M_+$ is defined as the determinant of a certain matrix with entries in 
the universal enveloping algebra of $\f R$ (which is a commutative ring).
See \cite{Bump-Choie}, p. 123.
Let us now write this out explicitly.  Let $\alpha(i,i)$ be the root such that
$2d_i^*$ and  for $i\neq j$ let $\alpha(i,j)$
be the root $d_i^*+d_j^*.$ 
(In particular, $\alpha(i,j)=\alpha(j,i)$). Then 
$$\hat M_+=\det( c_{\alpha(i,j)} X_{\alpha(i,j)}),$$
where $c_\alpha$ is defined as above.
Now, for each $i,j$ let $\hat M^{(i,j)}_+$ denote the $i,j$ cofactor 
in this determinant, 
(that is, the determinant of the appropriate minor, times the appropriate sign) 
which 
is also an element of the universal enveloping algebra of $\f R$.  

Identify $\fgl(n, \C)$ with $\Ad(c^{-1}) \fk$ via the map $X_0 \mapsto 
\bspm X_0 &0\\ 0 & -\,^\top X_0\espm.$
\begin{prop}
The space
\begin{equation}\label{span cofactors times u}
\Span\left(\{\hat M_+^{(i,j)} u : 1 \le i,j \le n\}\right)
\end{equation}
is nonzero, and it is an irreducible $\frak{gl}(n, \C)$ module.
Assume that 
\begin{equation}\label{span ys times w}
\Span\left(\{Y_\alpha w : \alpha \in \Phi(\fh, \fu^+)\}\right)
\end{equation}
is nonzero.
Then it is an irreducible $\frak{gl}(n, \C)$ module which is 
isomorphic to \eqref{span cofactors times u}.
\end{prop}
\begin{proof}
Since $w=\hat M_+u$ can be expressed as an $\f R$-linear combination 
of elements of \eqref{span cofactors times u}, the fact that \eqref{span cofactors times u}
is nonzero follows from the assumption that $w$ is nonzero. 
First, as $\mathfrak{sl}(n, \C)$-modules,   $\f R$ and $\f L$ are 
both isomorphic to the symmetric square representation, which 
is irreducible and self-dual.  The expansion of the determinant 
(an $\fsl(n, \C)$-invariant) by minors defines an $\fsl(n, \C)$-invariant 
bilinear form between  $\f R$ and the space of minors
\begin{equation}\label{space of minors in uea}
\Span\left(\{\hat M_+^{(i,j)} : 1 \le i,j \le n\}\right)\subset U( \fu^+).
\end{equation}
When regarded as $\fgl(n, \C)$-modules,  $\f L$ 
remains isomorphic to the dual, $\f R^*$, of $\f R,$  
while the $\fsl(n, \C)$-invariant bilinear form 
between \eqref{space of minors in uea} and $\f R$ spans 
a one dimensional $\fgl(n, \C)$-module on which $X_0 \in \fgl(n, \C)$
acts by $2\Tr X_0.$  It follows that \eqref{space of minors in uea}
is isomorphic, as a $\fgl(n, \C)$-module to the twist
of $\f R^*,$ or, equivalently, of $\f L,$ by the one dimensional representation $2\Tr.$ 

Now, if $x$ is any semispherical vector $x \in V$ of weight $l,$ the mapping 
$D \to Dx$ is a $\fsl(n, \C)$-equivariant linear map $U(\g) \to V.$  On each irreducible sub-$\fgl(n, \C)$-module $W$ of $U(\g)$ it is either trivial or injective, 
and when injective the image is a sub-$\fgl(n, \C)$-module of $V$ 
isomorphic to the twist of $W$ by $l \Tr.$
Thus \eqref{space of minors in uea} and \eqref{span ys times w}
are both isomorphic to the twist of $\f L$  by $(k+2r+2)\Tr.$
\end{proof}
\begin{cor}\label{cor: Y's and cofactors}
There is a scalar $C$ such that 
$$ Y_{\alpha(i,j)} w = C \hat M^{(i,j)}_+ u, \qquad (\forall 1 \le i,j \le n).$$ 
\end{cor}
\begin{proof}
If \eqref{span ys times w} is trivial then the statement holds with $C=0.$
Otherwise, comparing the $\fsl(n, \C)$-invariant bilinear forms on 
$\f R \times \f L$ and $\f R \times$\eqref{space of minors in uea}, 
one sees that the isomorphism \eqref{span ys times w}$\to$\eqref{span cofactors times u} must map $Y_{\alpha(i,j)} w$ to $C \hat M^{(i,j)}_+ u$
for each $i,j.$  Now, both \eqref{span ys times w} and \eqref{span cofactors times u} are contained in $U(\fg) v_0 = U(\f R)v_0,$ and 
it follows from proposition 4, p. 122 of \cite{Bump-Choie} that 
this space is multiplicity-free as a $\fgl(n, \C)$-module.  The result 
follows.
\end{proof}

\subsection{Completion of proof of \ref{U(g) version of Conjecture 2}}
In view of proposition 5 of \cite{Bump-Choie} (p. 124), it suffices to 
prove that the vector $w =\hat M_+^{r+1}$ satisfies \eqref{hol}, i.e., 
is annihilated by $\f L.$  Since 
$\f L$ is irreducible as a $\fsl(n, \C)$-module, and the 
mapping $X \to Xw$ is $\fsl(n, \C)$-equivariant, its kernel 
is either trivial or all of $\f L.$  Hence, it
 suffices to 
prove that $w$ is annihilated by a single element of $\f L.$  We shall
show that $w$ is annihilated by $\sum_{1\le i \le j \le n } Y_{\alpha(i,j)}.$
Suppose not.  Then, the constant $C$ in corollary \ref{cor: Y's and cofactors}
is nonzero.  It follows that 
$$
0 = 
\sum_{\alpha \in \Phi^+} c_\alpha X_\alpha Y_\alpha w = 
\sum_{1\le i\le j\le n} c_{\alpha(i,j)}
X_{\alpha(i,j)}Y_{\alpha(i,j)} w = C\sum_{1\le i\le j\le n} c_{\alpha(i,j)}X_{\alpha(i,j)}\hat M^{(i,j)}_+ u = C\hat M_+u = Cw,
$$
a contradiction.

\section{The other elements of the center}
\label{section: other elements of the center}
It is at first surprising that one may prove proposition \ref{U(g) version of Conjecture 2}, and hence conjecture 2 of \cite{Bump-Choie} using an
argument which references only the Laplace-Beltrami operator, and 
not any of the other elements of the center of $U(\g).$  
However, a hint as to why this is so is already available 
in theorem 3 of \cite{Bump-Choie}.  We briefly review this result.

The center of the universal enveloping algebra of $\f{sp}(4, \C)$ is a 
polynomial ring in two generators, 
one of degree $2$ and the other of degree $4.$  The degree $2$ 
generator is unique up to scalar, and the degree $4$ generator is 
unique up to scalar modulo the span of the square of the degree
$2$ generator.  Explicit choices for the two generators,
$\f D_2$ and $\f D_4,$ are fixed on p. 128 of \cite{Bump-Choie}, 
and it is shown in theorem 3 that 
 $\f D_4\equiv -2\f D_2,$ modulo a right ideal $\cJ$ which 
 annihilates any vector $v$ satisfying \eqref{wt(k)} and \eqref{hol}.

The next theorem generalizes this phenomenon.  
In order to state it, some notation is required.  
Let $\cJ$ denote the right ideal in $U(\fg)$ generated by $\fp.$
Let $\c Z$ denote the center of $U(\fg).$ 
Let $\f b$ denote the Borel subalgebra of $\fg$ consisting 
of all elements $\bspm X_0&0\\X&-\,^\top X_0 \espm \in \f p$ 
such that $X_0$ is lower triangular.  Let $\rho$ denote half the 
sum of the roots of $\f h$ in $\f b.$  
Let $\fh_0$ be the kernel of $\Tr$ in $\f h,$ and fix $H_0\in \fh \smallsetminus \fh_0.$ 

\begin{thm}
Let $\gm$ be the isomorphism $\c Z\to S(\f h)^W$ defined 
on p. 118 of \cite{HC}. 
Let $t: S(\fh)\to S(\fh)$ be the $\C$-algebra isomorphism
which is given on elements of $\fh$ by $H \mapsto H+\rho(H).$ 
let $\pr_1$ denote
projection onto the first factor in the canonical isomorphism 
$S(\fh) \to \C[H_0]\oplus \fh_0 S(\fh).$
Then $$\f D \equiv \pr_1\circ t \circ \gm(\f D) \pmod{\c J}\qquad \forall \f D \in \c Z.$$
\end{thm}
\begin{proof}
This follows from the construction of $\gm$ given on p. 118 of \cite{HC}.
Let $\c I$ denote the ideal in $U(\g)$ generated by the nilpotent 
radical of $\f b.$  
Then $\gm = t^{-1}\circ \gm'$ where $\gm'(\f D)$ is uniquely determined 
by the fact that $\gm'( \f D) \in U( \f h)$ and $\f D - \gm'(\f D) \in \c I$
for all $\f D \in U(\f g).$   Thus $\f D\equiv \gm'( \f D) = t\circ \gm(\f D)
\pmod{\c J}.$  Clearly, $\c J$ contains both $\c I$ and $\fh_0 S(\fh).$
\end{proof}
Thus, the action of $\c Z$ on $V$ factors through the projection 
to a quotient ring which is isomorphic to a 
polynomial algebra in one generator, and generated by the image of the Laplace-Beltrami operator.  

\section{Notation for the Jacobi case}
\label{section: jacobi case notation}
First, let us set up some notation.  The treatment of Jacobi forms in 
this paper is influenced by 
\cite{Berndt-Schmidt} as well as \cite{Bump-Choie} and \cite{Choie-Kim}.  Fix integers $n$ and $j$ for the remainder 
of the paper.  

For a ring $R$, let  $\sym^2_j(R)$ for the space of symmetric $j\times j$ matrices
with entries in  $R.$ 

\subsection{Jacobi group}
Define the Jacobi group $G^{(n,j)}$ to be the following subgroup
of $Sp(2n+2j)$:
$$
\left\{
\bpm I_n &0&0& \,^\top\!u\\ v&I_j&u& \zeta \\ 0&0&I_n& -\,^\top\!v\\
0&0&0&I_j\epm
\bpm A& 0&B&0 \\ 0&I_j&0&0\\C&0&D&0 \\ 0&0&0&I_j \epm:
u, v \in \Mat_{j\times n }, 
\zeta + u{}^\top v\in \sym^2_j,
\bpm A&B\\ C&D \epm \in Sp(2n)
\right\}.
$$
It is convenient to identify
$$
\bpm A&B\\ C&D \epm\in Sp(2n) \text{ with } \bpm A& 0&B&0 \\ 0&I_j&0&0\\C&0&D&0 \\ 0&0&0&I_j \epm \in G^{(n,j)}, 
\text{ and }
(v, u, \zeta) \text{ with }  
\bpm I_n &0&0& \,^\top\!u\\ v&I_j&u& \zeta \\ 0&0&I_n& -\,^\top\!v\\
0&0&0&I_j\epm.
$$ 
Doing so equips $H^{(n,j)}:\{(v, u, \zeta) \in \Mat_{j\times n} \times \Mat_{j\times n} \times \Mat_{j\times j}\;: \; \zeta + u{}^\top\!v\in \sym^2_j\}$
with the structure of a Heisenberg group, with operation
$$
(v_1, u_1, w_1) \cdot (v_2, u_2, w_2) = (v_1+v_2, u_1+u_2, 
w_1+w_2 + v_1{}^\top u_2 + u_1{}^\top v_2).
$$
\subsection{Symmetric space $\c H^{(n,j)}$}
Let $\c H_n$ denote the Siegel upper half-plane of genus $n,$ 
i.e., the space of all $\tau=X+iY \in \sym^2_n(\C)$ with $X, Y \in \sym^2_n(\R)$ 
and $Y$ positive definite.  The group $G^{(n,j)}(\R)$ 
acts on $\c H^{(n,j)}:=\c H^n\times \Mat_{j\times n}(\C)$ by the formulae
$$
\bpm A&B \\ C& D\epm \cdot ( \tau, z) =\left( (A \tau + B)(C\tau+D)^{-1}, z(C\tau+D)^{-1}\right), 
\qquad 
 (v,u, \zeta) \cdot ( \tau , z) = (\tau, z+u+\tau v).
$$
The stabilizer of $(iI_n, 0)$ is $U(n) \times Z(\R),$ where $Z=\{(0,0,\zeta): \zeta \in \sym^2_j\}$ is the 
center of $G^{(n,j)}.$

Given $\c M \in \sym^2_j(\C)$ define 
$e^{\c M}: \Mat_{j\times j}(\C) \to \C$ by 
$e^{\c M}(A) = e^{2\pi i \Tr(\c M A)}.$  Observe that $e^{\c M}( A) = e^{\c M}({}^\top A)$ for all $A \in  \Mat_{j\times j}(\C).$

\subsection{Covering groups}
Let $\Sp(2n, \R)$ be the metaplectic double 
cover of $Sp(2n, \R).$  
Choose a branch of the square root, and define 
$$
J_{\frac12}\left( \bpm A&B\\C&D \epm , \tau \right):= \det ( C\tau +D)^{-\frac12},
\qquad \sigma( g_1, g_2) = \frac{J_{\frac12}( g_1g_2, \tau)}{J_{\frac12}(g_1, g_2\tau)J_{\frac12}(g_2, \tau)}.
$$
Then one may identify 
$
\Sp(2n, \R)$ 
with the set $Sp(2n, \R)\times \{\pm 1\},$
equipped
 with multiplication 
$(g_1, \ve_1) (g_2, \ve_2) = (g_1g_2, \sg(g_1, g_2)\ve_1\ve_2),$
as in \cite{Bump-Choie}.
However, the topology on $\Sp(2n, \R)$ is such that
multiplication is continuous, and hence 
disagrees with that of $Sp(2n, \R)\times \{\pm 1\}.$

Define an action of $\Sp(2n, \R)$ on $H^{(n,j)}_\R$ by composing
the usual action of $Sp(2n,\R)$ with the projection $\Sp(2n,\R) \to Sp(2n, \R),$ and let $\wt G^{(n,j)}_\R$ be the semidirect product
$\Sp(2n, \R) \rtimes H^{(n,j)}_\R,$ defined using this action.  The function
$$
J_{\frac 12}\left(\left(\bpm A&B\\ C&D \epm, \ve\right), 
 \; \tau\right) = J_{\frac12}\left( \bpm A&B\\C&D \epm , \tau \right)\ve
$$
satisfies the cocycle condition 
$$
J_{\frac12}(g_1g_2, \tau) 
= J_{\frac12}(g_1, g_2\tau)J_{\frac12}(g_2, \tau),
\qquad (\forall g_1, g_2 \in \Sp(2n, \R), 
\; \tau \in \c H_{n}).
$$

\subsection{Cocycle and slash operator}

Fix $k \in \Z$ and $\c M \in \sym^2_j\frac12\Z.$ For $\left(\bspm A&B\\ C&D \espm, \ve\right) \in \Sp(2n, \R),$ 
$(v,u,\zeta) \in H^{(n,j)}_\R,$ and $(\tau , z) \in \c H^{(n,j)},$ define 
 $$\begin{aligned}
&j_{\frac k2, \c M}\left(g (v,u,\zeta), (\tau, z)\right)
\\&\qquad=J_{\frac 12}\left(g, 
 \; \tau\right)^k
e^{\c M}( \zeta + (v \tau + 2 z + u){}^\top v - (z+u+v\tau)(C\tau+D)^{-1}C (z+u+v\tau)),
\end{aligned}
$$
Then $j_{\frac k2, \c M}$ satisfies the cocycle condition
$$
j_{\frac k2, \c M}( g_1g_2 , \b x) = j_{\frac k2, \c M}(g_1, g_2\b x) j_{\frac k2, \c M}( g_2, \b x) 
\qquad ( \forall g_1, g_2 \in \wt G^{(n,j)}_\R, \b x \in \c H^n \times \Mat_{j\times n}(\C)).
$$
Now, for $f: \c H^{(n, j)} \to \C$ and $g \in \JC_\R$ define 
$f|_{\frac k2, \c M}g: \c H^{(n,j)} \to \C$ by 
$$\left(f\Big|_{\frac k2, \c M}g \right)(\b x) := j_{\frac k2, \c M}( g, \b x) 
f( g \b x), \qquad 
\left(\b x \in \c H^{(n,j)}\right).$$

\subsection{Meromorphic Jacobi forms}
Let $j^0_{\frac k2, \c M}: \JC_\R \times  \c H^{(n,j)} \to \C$ be defined by the 
same formula as $j_{\frac k2, \c M}$ with the factor $\ve ^k$ omitted.
To define Jacobi forms, one chooses a discrete subgroup $\Gm \subset \Ja_\R,$ and a function $\chi: \Gm \to \C$ such that 
$(\gm, \b x) \mapsto \chi(\gm) j_{\frac k2,\c M}^0(\gm, \b x)$ is a cocycle.
In this case $\wt\chi((\gm, \ve)):= \chi( \gm) \ve$ is a character
of the preimage $\wt \Gm$ of $\Gm$ in $\JC_\R.$
Then a meromorphic Jacobi form of weight $\frac k2,$ index $\c M,$ and multiplier
$\chi$  is a meromorphic function 
$f: \c H^{(n,j)} \to \C$
which satisfies 
\begin{equation}\label{modularity property of Jacobi forms}
j_{\frac k2, \c M}^0(\gm, \b x) \chi(\gm) f( \gm \b x) = f(\b x) \qquad \left(\forall
\gm \in \Gm, \; \b x \in \c H^{(n,j)} \right).
\end{equation}
One might equivalently describe it as being invariant under the 
right-action of $\wt\Gm$ 
on functions $\c H^{(n,j)} \to \C$ 
defined by 
$$\left(f\Big|_{\frac k2, \c M,\wt \chi}\gm \right)(\b x) := j_{\frac k2, \c M}( \gm, \b x) 
\wt\chi( \gm ) f( \gm \b x), \qquad 
\left( \gm \in \wt \Gm, \; \b x \in \c H^{(n,j)}\right).$$

\section{Main Result, a la Choie-Kim}
\label{section: Main result a la Choie-Kim}

\begin{defn}
Take $\c M \in \sym^2_j (\R),$ positive definite.  Define 
$\ptl{}\tau$ to be the $n \times n$ matrix with $r,s$ entry $(1+\delta_{r,s})\ptl{}{\tau_{rs}},$ and $\ptl{}z$ the $j \times n$ matrix with $r,s$ entry 
$\ptl{}{z_{rs}}.$  Finally, write $|\c M|$ for the determinant of $\c M$ 
and $\wt{\c M}$ for the classical adjoint of $\c M,$ i.e., the matrix of cofactors satisfying $\c M^{-1} = |\c M|^{-1} \wt{\c M}.$ Then 
$$
L_{\c M}:= \det \left( 4 \pi i |\c M| \ptl{}\tau -{\ptl{}z}^{\hskip -14pt t \hskip 14pt} \wt{\c M} \ptl{}z\right).$$ 
\end{defn}

\begin{thm}\label{Main thm: jacobi}
Let $f$ be a meromorphic Jacobi form of weight $-r 
+ \frac{n+j+1}2$
and index $\c M.$ 
  Then 
$L_{\c M}^k \cdot f$ is a meromorphic 
Jacobi form of weight $r 
+ \frac{n+j+1}2$ and index $\c M.$ 
\end{thm}
\begin{proof}  Theorem \ref{Main thm: jacobi} can 
actually be deduced from proposition \ref{U(g) version of Conjecture 2} and the results of \cite{Bump-Choie}.
First, define the {\bf extension by $\pmb{e^\c M}$} of $f$
$$\on{ext}_{\c M}f
\bpm \tau & ^\top z\\ z&\tau' \epm 
= f(\tau, z) e^\c M(\tau'),
\qquad ( \tau \in \c H_n, z \in \Mat_{j\times n}\C, 
\tau' \in \c H_j).
$$
Clearly, meromorphicity of $f$ implies that
of $\on{ext}_{\c M} f.$  Further, it follows from a determinant identity given on 
p. 86 of \cite{Choie-Kim} that 
$\D^l \on{ext}_\c M f = [(4\pi i)^{n-j} \det \c M^{n-1}]^l
\on{ext}_{\c M} L_{\c M}^l f$
for all $l \in \N.$
Next, define 
$$\wt \sigma_k: C^\infty(\c H_{n+j}) 
\to C^\infty (\wt Sp(2n+2j, \R)).$$
By 
$$\begin{aligned}
\wt \sigma_k f\left(g
\right)
&=
\left[f\big|_k g\right]\left(iI_{n+j}\right)=
J_{\frac12}(g, iI_{n+j}) f\left(g\cdot iI_{n+j}\right).
\end{aligned}
$$
Then for any $f_1 \in C^\infty( \c H_{n+j}),$ 
$$
\wt \sg_k f_1 = \sg_k F_1, \qquad \text{ where }
F_1(\tau) = \det Y^{k/2} f(\tau),
$$
and $\sg_k$ is defined as in \cite{Bump-Choie}.
The image of $\sg_k$ is the space
$C^\infty( \Sp(2n+2j, \R))_{(\wt U(n+j), \det^k)}$ 
consisting of all $\phi \in C^\infty( \Sp(2n+2j, \R))$
which satisfy
$$
\phi\left( g \kappa \right) 
= \phi(g) J_{\frac12}( \kappa , iI_{n+j})^{2k} \qquad \left(\forall g \in \Sp(2n+2j , \R), \kappa \in \wt U(n+j)\right).
$$
If $k \in \Z$ this takes the more convenient form 
$$
\phi\left( g \bpm A&-B\\ B& A\epm \right)
= \phi(g) \det(A+Bi)^k 
\qquad( \forall g \in Sp(2n+2j, \R), \; A+Bi \in U(n+j)).
$$
As shown in \cite{Bump-Choie} the action of
$M_+ \in U( \frak{sp}(2n+2j, \C))$ 
on $C^\infty( \Sp(2n+2j, \R))$ 
maps $C^\infty( \Sp(2n+2j, \R))_{(\wt U(n+j), \det^k)}$
to $C^\infty( \Sp(2n+2j, \R))_{(\wt U(n+j), \det^{k+2})}$
for each $k \in \frac 12\Z.$
This
induces 
operators $$\M_k:= \wt \sg_{k+2}^{-1} \circ 
M_+ \circ \wt \sg_k: C^\infty( \c H _{n+j})
\to C^\infty( \c H _{n+j}),\qquad 
\left(k \in \frac 12\Z\right).$$
One may also define 
$\M_k^l:= \wt \sg_{k+2l}^{-1} \circ 
M_+^l \circ \wt \sg_k
= \M_{k+2l-2} \circ \dots \circ \M_{k+2}\circ \M_k$
for each $k \in \frac 12 \Z, \; l \in \N.$

 \begin{lem}\label{M is basically D}
  Suppose 
 that $f\in C^\infty(\c H_{n+j})$ and 
 $\M_k^lf$
 are both meromorphic.  
 Then 
 $$
 \M_k^l f = 2^{ln} \D^l f.
 $$
 \end{lem} 
 \begin{proof}
We expand slightly on the proof given on p. 126 of 
\cite{Bump-Choie}.  Write $\tau \in \c H_n$
as $X+iY$ and say that a function is ``simple nonmeromorphic'' if it is equal to a meromorphic 
function times a rational function of $Y$ with the 
property that the degree of the denominator exceeds
that of the numerator.  
Note that the zero function is both meromorphic 
and simple nonmeromorphic, and is the only function 
with both properties.

It follows from results of Harris \cite{Harris} that the 
operator 
$\M_k$ can be realized explicitly as 
 multiplication by $(\det Y)^{-\frac{n+j-1}2+k}$
 followed by $2^n \D,$ and then multiplication 
 by $(\det Y)^{\frac{n+j-1}{2}-k}.$
 This implies that $\M_k f - 2^n \D f$ is a simple nonmeromorphic function for any meromorphic 
 function $f,$ and that $\M_k$ applied to any 
 simple nonmeromorphic function is again
 simple nonmeromorphic.  Thus, for $f$ meromorphic
 $\M_k^lf - 2^{nl}\D^l f$  is simple nonmeromorphic.  If 
 $\M_k^lf$ is meromorphic, then $\M_k^lf - 2^{nl}\D^l f$
 is also meromorphic, and therefore zero.
 \end{proof}

Let 
$\Mer( \c H_{n+j})$ denote the space of 
meromorphic functions $\c H_{n+j} \to \C$ and let 
$$\Mer(\Sp(2n+2j, \R))_{(\wt U(n+j), \det^k)}:= \wt \sg_k \left[\Mer( \c H_{n+j})\right]\subset C^\infty (\Sp(2n+2j, \R))_{(\wt U(n+j), \det^k)}
.$$  Now,
as explained in \cite{Bump-Choie}, pp. 124-27, it follows from proposition \ref{U(g) version of Conjecture 2} that 
$\M_{-r+\frac{n+j+1}2}^r$ maps 
$\Mer(\c H_{n+j})$ to $\Mer(\c H_{n+j}),$ and therefore
coincides with a scalar multiple of $\D^r$ on $\Mer(\c H_{n+j}).$

Now, 
let $\Lambda$ denote the right-action of 
$\Sp(n, \R)$ on $C^\infty(\Sp(n, \R)$ 
given by 
$[\Lambda(g)f](h) = f(gh).$ 
A straightforward calculation shows that
$$
\left(\on{ext}_\c M f_1\Big |_k g\right)
= \on{ext}_\c M \left( f_1\big|_{k, \c M} g\right),
\qquad 
\left(\forall g \in \JC_\R, f_1 \in C^\infty( \c H^{(n,j)})\right).
$$
So
$$
f \big|_{k, \c M} \gm = \wt \chi(\gm) f 
\iff 
(\on{ext}_\c M f)\big|_k \gm = \wt \chi( \gm) f
\iff \Lambda(\gm) \wt \sg_k(\on{ext}_\c M f) = \wt\chi(\gm ) \wt \sg_k(\on{ext}_\c M f).
$$
Now it's clear that the action of $\Sp(2n+2j, \R)$ on 
$C^\infty( \Sp(2n+2j, \R))$ via $\Lambda$
commutes with the action of $U( \f{sp}(2n+2j, \C))$
on the right.  Hence
$$
\Lambda(\gm) M_+^r \wt \sg_k(\on{ext}_\c M f) = \wt\chi(\gm ) M_+^r \wt \sg_k(\on{ext}_\c M f),
$$
and so
$$\left[(\M_k^r \on{ext}_\c M f)\Big|_{k+2r} \gm \right]
=\wt\chi(\gm) (\M_k^r \on{ext}_\c M f).$$
If $k=-r+\frac{n+j+1}2,$ then 
proposition \ref{U(g) version of Conjecture 2}
implies that $(\M_k^r \on{ext}_\c M f)$ is meromorphic,
and therefore 
one may replace $\M_k^r$
be $2^{nr}\D^r.$ Then the result follows from 
the relationship between $\D, \on{ext}_\c M$ 
and $L_{\c M}$ already mentioned.
\end{proof}
\section{Alternate approach, a la Bump-Choie}
\label{section:  Alternate approach, a la Bump-Choie}
In the last section we showed that theorem 
\ref{Main thm: jacobi}, which is essentially a version
of Bol's result for Jacobi forms, may be proved
by extending Jacobi forms to functions on $\c H_{n+j}$ 
which satisfy a certain equivariance property and
applying the same results used to prove Bol's result for Siegel modular forms, Conjecture 1 of \cite{Bump-Choie}.

One may ask whether it is also possible to prove
theorem \ref{Main thm: jacobi} by developing 
a version of proposition \ref{U(g) version of Conjecture 2}
for Jacobi forms, which reduces theorem \ref{Main thm: jacobi} to a statement regarding actions of the Lie algebra
of the Jacobi group, and then proving that statement.
In this section, we 
perform a translation for Jacobi forms, 
similar to the one given in 
\cite{Bump-Choie} for Siegel modular forms, and then 
state and prove a result which 
can be regarded as a version of \ref{U(g) version of Conjecture 2}
for Jacobi forms.  This approach will fall short of a second 
proof of theorem \ref{Main thm: jacobi}, because no 
analogue of lemma \ref{M is basically D}
will be proved.
In other words, we shall show that a certain differential
operator which arises naturally from consideration 
of the universal enveloping algebra exhibits the ``recovery'' property characteristic of Bol's result, but we shall not 
prove a result relating this operator back to the more
classical heat operator $L_{\c M}$ considered in \cite{Choie-Kim}.

\subsection{Liftings to the Jacobi group}
For $k \in \frac 12 \Z,$ 
 $\c M \in \sym^2_j \frac 12 \Z,$
and $f: \c H^{(n,j)} \to \C,$ 
define 
$$
\left[\varphi_{k, \c M} f\right](g) :=\left(f\Big|_{k, \c M}g \right) (iI_n, 0).
$$
Clearly, for any function $f,$ the function $\left[\varphi_{k, \c M} f\right]$
satisfies
\begin{equation}\label{weight k index M equivariance property}
\begin{aligned}
\left[\varphi_{k, \c M} f\right]\left(g (0,0,\zeta) \bpm A&B\\-B&A\epm  \right) 
= e^{\c M}( \zeta) \ve^{2k}&\det(A+iB)^{k} \left[\varphi_{k, \c M} f\right](g), \\
&\left(\forall g \in \JC_\R, \; \zeta \in \sym^2_j(\R), \; (A+iB,\ve)\in \wt U(n)\right).
\end{aligned}\end{equation}
Moreover, if $f$ satisfies \eqref{modularity property of Jacobi forms},
then $\varphi_{k, \c M} f$ satisfies 
$$[\varphi_{k, \c M} f](\gm g) =\wt\chi^{-1}(\gm) [\varphi_{k, \c M} f](g), \qquad ( \forall \gm \in \wt \Gm, \; g\in \JC_\R).$$

\subsection{Relating holomorphicity to the action of the Lie algebra}
Let $\jla$ denote the Lie algebra of $\Ja.$  It consists of all matrices of 
the form 
$$
\bpm A&0&B&^\top u\\ v&0&u&Z \\ C&0&-^\top \!A&-^\top v\\ 0&0&0&0 \epm,
\qquad A \in \Mat_{n\times n}, \; B,C \in \sym^2_n, \; u, v \in \Mat_{j\times n},
\; Z \in \sym^2_j.
$$
Identify 
$$\bpm A&0&B&0\\ 0&0&0&0 \\ C&0&-{}^\top \!A&\\ 0&0&0&0 \epm
\text{ with } \bpm A&B \\ C& -{}^\top \!A \epm \in \f{sp}(2n), \text{ and }$$ 
$$\bpm 0&0&0&^\top u\\ v&0&u&Z \\ 0&0&0&-^\top \!v\\ 0&0&0&0 \epm
\text{ with } [u, v, Z] \in \Mat_{j\times n}\times  \Mat_{j\times n}
\times \sym^2_j,
$$
thus inducing a bracket on $\f h^{(n,j)}:= \Mat_{j\times n}\times  \Mat_{j\times n} 
\times \sym^2_j,$
$$\Big[[u_1, v_1, Z_1], [u_1, v_2, Z_1]\Big]
= [0,0,-u_1\,^\top \!v_2-v_2\,^\top \!u_1+u_2\,^\top \!v_1+v_1\t u_2].$$
(A triple in $\f h^{(n,j)}$ will be written with brackets to avoid being mistaken 
for an element of $\c H^{(n,j)}.$)

Next, let 
$$
c_n = \bpm \frac{1}{1+i}I_n & \frac{-i}{1+i}I_n\\ &\\ \frac{1}{1+i}I_n &  \frac{\phantom{-}i}{1+i}I_n\epm.
$$
be the Cayley transform for $Sp(2n)$ (which we
identify with an element of $\Ja_\C$)
and define an action of $U(\jla_\C)$ on smooth functions $\JC_\R \to \C$ 
by twisted right translation 
\begin{equation}\label{twisted right translation}
R^{c_n}(X)\phi(\wt g):= \lim_{t\to 0} \frac1t\left(\phi( \wt g \exp(t c_n^{-1}Xc_n)) -\phi(\wt g)\right), \qquad c_n^{-1}Xc_n\in \jla_\R, 
\end{equation}
extended to all of $\jla_\C$ by $\C$-linearity.
\begin{prop}\label{holomorphicity in terms of action of lie algebra-- jacobi}
A smooth function $f: \c H^{(n,j)}\to \C$ is holomorphic if and only if the
corresponding function $\varphi_{k, \c M} f$ is annihiliated by $R^{c_n}(X)$
for all $X$ of the form 
$$
\bpm 0&0\\ X_0 & 0 \epm, X_0 \in \sym^2_n (\R), \qquad 
[\lambda,0,0], \; \lambda \in \Mat_{j\times n}(\R). 
$$
\end{prop}
\begin{proof}
Write $\b x \in \c H^{(n,j)}$ as $(\tau, Z)$ with $\tau \in \c H^n$
and $Z \in \Mat_{j\times n}(\C).$  The first goal is to show that 
$f: \c H^{(n,j)} \to \C$ is holomorphic in the variable $Z$ if and 
only if $\varphi_{k, \c M}f$ is annihilated by $R^{c_n}([\lambda, 0,0])$ for 
every $\lambda \in \Mat_{j\times n}(\R).$
Write $Z=(z_{rs})_{r=1}^j{}_{s=1}^n \in \Mat_{j\times n}(\C)$ as
$U+iV,$ where $U = (u_{rs})_{r=1}^j{}_{s=1}^n$ and $V= (v_{rs})_{r=1}^j{}_{s=1}^n$ are in $\Mat_{j\times n}(\R).$
A function is holomorphic in $Z$ if and only if it 
satisfies: 
\begin{equation}\label{hol in Z variant 1}
\ptl{f}{z_{rs}} = \ptl{f}{u_{rs}} +i\ptl{f}{v_{rs}}=0\qquad  (\forall 1\le r \le j, \; 1\le s \le n).
\end{equation}
Form the matrix $\ptl fZ = (\ptl f{z_{rs}})_{r=1}^j{}_{s=1}^n.$  Then \eqref{hol in Z variant 1} is equivalent to 
\begin{equation}\label{hol in Z variant 2}
\ptl fZ =0, \quad \text{ or } \ptl fZ \cdot A = 0 \text{ for any given }A \in GL(n, \C).
\end{equation}
Now let $\lambda_{rs}^* \in \f h^{(n,j)}_\R$ be the element $[e_{rs},0,0]$
where $e_{rs}$ is the matrix with a $1$ at the $rs$ position and 
zeros everywhere else.  Also define $\mu_{rs}^*$ to be
$[0,e_{rs},0].$
Define the usual (untwisted) action of $\jla_\R$ on functions $\JC_\R\to \C$
by 
$$
R(X)f(\wt g) = \ddt f(\wt g \exp(tX))\Big|_{t=0}.
$$
One checks that 
$$\begin{aligned}
(\lambda, \mu, \zeta) g_0 \exp( t\lambda_{rs}^*)\cdot (i,0)
=\left(g_0\cdot i, \mu + \lambda g_0 \cdot i+i t \lambda_{rs}^* (Ci+D)^{-1}\right),
\\
(\lambda, \mu, \zeta) g_0 \exp( t\mu_{rs}^*)\cdot (i,0)
=\left(g_0\cdot i, \mu + \lambda g_0 \cdot i+ t \mu_{rs}^* (Ci+D)^{-1}\right),
\end{aligned}
$$
for all $(\lambda, \mu, \zeta) \in H^{(n,j)}_\R, g_0 = \bspm *&*\\C&D \espm
\in \Sp(2n, \R), \, t\in \R, \; 1\le r \le j,\; 1\le s \le n.$
It follows that \eqref{hol in Z variant 2}, with $A = (Ci+D)^{-1},$
is equivalent to 
\begin{equation}\label{hol in Z variant 3}
R(\mu_{rs}^*)f+iR(\lambda_{rs}^*)f = 0, 
\qquad (\forall\; 1\le r \le j,\; 1\le s \le n).
\end{equation}
One checks that $c_n(\mu_{rs}^*+ i \lambda_{rs}^*)c_n^{-1}
= -2 \lambda_{rs}^*,$ and this proves that 
$f$ is holomorphic in $Z$ if and only if it is annihilated
by $R^{c_n}( \lambda_{rs}^*)$ for all $1\le r \le j,\; 1\le s \le n.$

Suppose this to be the case.  Then $f$ is holomorphic 
on $\c H^n \times \Mat_{j\times n}(\C)$ if and only if 
$\tau \mapsto f( \tau, \mu + \lambda\tau)$ is holomorphic 
for each fixed $\lambda,\mu \in \Mat_{j\times n}(\R).$  Thus, 
the proposition is reduced to its analogue from the Siegel case,
proved in \cite{Harris}, \S 2.3.1 (cf. \cite{Bump-Choie}, p. 126).
\end{proof}
\subsection{A Jacobi analogue of proposition 
\ref{U(g) version of Conjecture 2}}
\begin{thm}\label{Jacobi U(g) version of Conjecture 2}
Let $V$ be a $\jla_\C$-module.
Say that $v \in V$ is holomorphic if 
\begin{equation}\label{J-holomorphic}\tag{Hol - Jac}
[\lambda ,0,0]v = \bpm I&0\\ X& I \epm 
v = 0,
\end{equation}
that $v$ is of index $\c M$ if
\begin{equation}\label{index M} \tag{ind $\c M$}
[0,0,\zeta]\cdot v = 2\pi i \Tr(\zeta \c M) \cdot v
\qquad \forall \zeta \in \Mat_{j\times j}\C.
\end{equation}
and that $v$ is of weight $k$ if 
\begin{equation}\label{weight k}\tag{weight $k$ - Jac}
\bpm A&0 \\  0&-{}\,^\top \!A\epm
v =  k\Tr A v, \qquad
(\forall A \in \f{gl}(n, \C)).
\end{equation}
Suppose that $v_0$ is holomorphic of 
index $\c M$ and weight $(-r+\frac{n+j+1}2).$
Then 
$
\hat M_+^r  v_0
$
is holomorphic of index $\c M$ and weight $(r+\frac{n+j+1}2).$  Moreover, if $k \ne r$ then 
$\hat M_+^k v_0$ is of weight $(-r+\frac{n+j+1}2+2k)$
and index $\c M,$ but not holomorphic.
\end{thm}
\section{Proof of theorem \ref{Jacobi U(g) version of Conjecture 2}}

\begin{defn}
Let $\f v$ denote the subspace
of $\f h^{(n,j)}$ consisting 
of all matrices of the form 
$$
\bpm 0&0&0&^\top u\\ v&0&u&0 \\ 0&0&0&-^\top \!v\\ 0&0&0&0 \epm,
$$
and let $\f z$ denote the center of $\jla.$
\end{defn}
\begin{defn}
Define $T_0: \f h^{(n,j)} \to \f{sp}(2n)$ by 
$$\begin{aligned}
T_0
\bpm 0&0&0&^\top u\\ v&0&u&Z \\ 0&0&0&-^\top \!v\\ 0&0&0&0 \epm
&=\bpm 0&0&0&^\top u\\ v&0&u&0 \\ 0&0&0&-^\top \!v\\ 0&0&0&0 \epm
\bpm
 0&0&0&0\\
 0&0&0&0\\
 0&0&0&0\\
 0&\wt Z&0&0
\epm
\bpm 0&0&0&^\top u\\ v&0&u&0 \\ 0&0&0&-^\top \!v\\ 0&0&0&0 \epm\\
&=\bpm ^\top u\wt Z v &0&^\top u\wt Z u&0\\ 0&0&0&0 \\ -^\top \!v\wt Z v &0&-^\top \!v\wt Z u&0\\ 0&0&0&0 \epm,
\end{aligned}
$$
where $\wt Z$ is the ``classical adjoint''
of $Z.$ 
\end{defn}
Observe that $T_0$ is 
 $\f{sp}(2n)$-equivariant.  Furthermore, it is nonzero.
 As the adjoint representation of $\f{sp}(2n)$ is 
 irreducible, it follows that $T_0$ is surjective.
 This gives an injective, $\f{sp}(2n)$-equivariant 
 map from  the dual of $\f{sp}(2n)$ into the space of polynomial functions on $\f h^{(n,j)}.$
 One may identify $\jla$ with its dual using the 
 invariant bilinear form $(X, Y) \mapsto \Tr(XY).$  
 Then each of the spaces 
  $\f{sp}(2n),$ $\f v$ and $\f z$  is identified with its dual 
  and 
 we obtain an injective $\f{sp}(2n)$-equivariant map
 $T:\f{sp}(2n)\to \sym^2\f v \otimes \sym^{j-1} \f z.$
 
The map $T$ may be described concretely as follows.
 Define $\b Z\in \sym^2_j(\f z)$ to be the $j \times j$ matrix with $i,l$ entry equal to $
(E_{n+i,2n+j+l}+E_{n+l,2n+j+i}) \in \f z,$
define $\b V \in \Mat_{j\times n} \f v$ to be the 
$j \times n$ matrix with $i,l$ entry equal to
$E_{n+i, l}-E_{n+j+l, 2n+j+i}\in \f v,$ and define $\b U \in \Mat_{j\times n} \f v$ to be the 
$j \times n$ matrix with $i,l$ entry equal to
$E_{n+i, n+j+l}+E_{l,2n+j+i} \in \f v.$
Then $$T(X) = \Tr\left(\tr X
\bbm \tr\b U\\ \tr \b V
\ebm 
\wt{\b Z}
\bbm \b V& \b U\ebm 
 \right).$$
 
   \begin{prop}\label{prop: invariance of T}
  For all $V \in \f v$ and all $X \in \f{sp}(2n)$ one has
  $$
  \ad(V)T(X) = 2\det \b Z \cdot  [V,X].
  $$
 \end{prop} 
 \begin{proof}
 The space
 $$
 \{ X \in \f{sp}(2n): 2 \det \b Z \cdot [V,X] 
 - ad(V)\cdot T(X)  = 0 \; \forall V \in \f v\}
 $$
  is clearly an $ \f{sp}(2n)$-submodule.
  We claim that it is $\f{sp}(2n).$  
It suffices to show that it is nonzero.
  The space $\f v$ has a basis 
  $$\{E_{n+i,l}-E_{n+j+l,2n+j+i}: 1 \le i \le j, \; 1 \le l \le n\}
  \cup \{E_{n+i,n+j+l}+E_{l,2n+j+i}: 1 \le i \le j, \; 1 \le l \le n\}.
  $$
  It suffices to consider $V$ in this basis.  We henceforth 
  assume $V\in \f v$ is in this basis.
  
 Take 
$X = E_{1,n+j+1}.$  
Then $$[V,X]=\begin{cases}
E_{n+i,n+j+1}+E_{1,2n+j+i}
&
V=E_{n+i,1}-E_{n+j+1, 2n+j+i} \; (1\le i \le j)\\
0, & V\text{ not of this form.}
\end{cases}$$
One may express $T(X),$ 
as 
$
\b u \widetilde{\b Z} \,^\top\!\b u ,
$
where
$$\b u= [E_{n+1,n+j+1}+ E_{1, 2n+j+1}, \dots , E_{n+j, n+j+1}+E_{1, 2n+2j}] \in \Mat_{1\times j}
\f v,$$
and the product should be computed in the symmetric
algebra of $\f h^{(n,j)}.$
Now, $V$ acts trivially on $\f z,$ so 
$$
\ad(V)\left(\b u\widetilde {\b Z} \,^\top\!\!\b u\right)
= \left(\ad(V)\b u\right) \widetilde  {\b Z}  \,^\top\!\! \b u + \b u \widetilde {\b Z} \left(\ad(V)\,^\top\!\!\b u\right)
$$
Further, 
$$
\ad(V)U = 
\begin{cases}
[E_{n+1, 2n+j+i}+ E_{n+i, 2n+j+1}, \dots , E_{n+j, 2n+j+i}+E_{n+i,2n+2j}], & 
V=E_{n+i,1}-E_{n+j+1, 2n+j+i}\\& (1\le i \le j)\\
0, & V\text{ not of this form.}
\end{cases}$$
In the case when $\ad(V)\b u$ is nonzero it agrees
with the $i$th row/column of the symmetric 
matrix $\b Z.$  
It follows that 
$\ad(V)\b u\wt {\b Z}$ is equal to $\det \b Z$ times the 
$i$th standard basis vector.
Hence, $\left(\ad(V)\b u\right) \widetilde  {\b Z}  \,^\top\!\! \b u$ equals $\det \b Z$ times the $i$th entry of $\b u,$
which is $E_{n+i,n+j+1}+E_{1,2n+j+i}.$  It follows 
by symmetry that 
$\ad(V)T(X) = 2 \det \b Z (E_{n+i,n+j+1}+E_{1,2n+j+i}).$
     \end{proof}
\begin{defn}
Define $\b T: \f{sp}(2n)\to U(\jla)$
by 
$\b T(X) = \lambda(2\det \b Z\cdot X - T(X))= 2\det \b Z\cdot X -\lambda( T(X)).$  Here, 
$\lambda$ denotes symmetrization as on \cite{Bump-Choie}, p. 128.
\end{defn}
It follows from proposition \ref{prop: invariance of T}
that the image of $\b T$ commutes with
$U(\f h^{(n,j)}) \subset U(\jla).$
\begin{prop}
One has $$
\b T(X) \b T(Y)- \b T(Y) \b T(X) =: [\b T(X) ,\b T(Y)] 
=2 \det \b Z\b T([X,Y]).
$$
\end{prop}
\begin{proof}
Indeed, 
$$\begin{aligned}
2\det \b Z\b T([X,Y]) &= 4(\det \b Z)^2 [X,Y] - \det \b Z\lambda(\ad(X)T(Y)) \\
&=4 (\det \b Z)^2 [X,Y] - 2\det \b Z[X,\lambda(T(Y))] 
\\
[\b T(X) ,\b T(Y)] 
&= [2\det \b Z\cdot X -\lambda( T(X)), \b T(Y)]
= 2\det \b Z [X,  \b T(Y)]\\
&=4( \det \b Z)^2 [X,Y] - 2\det\b Z
[X,\lambda(T(Y))].\end{aligned}
$$
\end{proof}
\begin{defn}\label{def:  the action *}
Let $V$ be a $\jla$-module such that 
$[0,0,\zeta]\cdot v = 2\pi i \Tr(\c M \zeta)  v$
for all $\zeta \in \Mat_{j\times j},$ symmetric,
and all $v \in V.$
Define $X * v = \frac1{2\det \c M} \b T(X)v.$
The previous proposition shows that 
$X*Y*v-Y*X*v=[X,Y]*v.$  Hence, $*$ extends to an 
action of $U(\f{sp}(2n))$ on $V.$  
\end{defn}
\begin{lem}\label{lem:  holomorphic for . implies holomorphic for *}
Let $V$ be a $U(\jla_\C)$-module with action 
$\cdot$ and define an alternate action of $\f{sp}(2n, \C)$
by $*$ as above.  Suppose $v\in V$
satisfies
$$
\bpm 
0&0\\X&0 
\epm \cdot v = [w, 0,0]\cdot v =0,
\forall 
X \in \Mat_{n\times n}\C, \text{ symmetric, }
w \in \Mat_{j\times n}\C.
$$
Then 
$$
\bpm 
0&0\\X&0 
\epm * v = 0 \forall 
X \in \Mat_{n\times n}\C, \text{ symmetric.}
$$
\end{lem}
\begin{proof}
Write $\f v = \f l + \f r$
where 
$\f l = \{ [w,0,0], w \in \Mat_{j\times n}\}$
and $\f r =  \{ [0,\mu,0], \mu \in \Mat_{j\times n}\}.$
Then it follows from the definition of $T$ that
$$T\bpm 
0&0\\X&0 
\epm 
\in \sym^2 \f l \subset \sym^2 \f v,
$$
and the result follows.
\end{proof}
\begin{lem}\label{lemma relating . and * for Siegel Levi}
Fix $\c M \in \Mat_{j\times j}\Z,$ symmetric.
Let $V$ be a $U(\jla_\C)$-module such that 
$[0,0,\zeta]\cdot v= 2\pi i \Tr( \c M \zeta )v$
for all $\zeta\in \Mat_{j\times j}\C$ and $v \in V.$ 
Let $\cdot$ denote the given action of $\jla_\C$
on $V$ and define a second action $*$ of $\f{sp}(2n, \C)$
as in definition \ref{def:  the action *}.
Suppose $v_0\in V$ satisfies
$$
[w, 0, 0]\cdot v =0\qquad \forall w
\in \Mat_{j\times n}\C.
$$
Then 
$$
\lambda\left( T\bpm X\\ &^\top \!X \epm \right)
\cdot v = \frac{j}2 \Tr X\cdot v,
$$
whence
$$
\bpm X\\ &^\top \!X \epm 
\cdot v = k\Tr(X)v 
\implies \bpm X\\ &^\top \!X \epm 
* v = (k-\frac j2)\Tr(X)v 
$$
\end{lem}
\begin{proof}
Write $V_{2n}$ for the standard representation
of $Sp(2n),$ realized as column vectors, and 
write $U_j$ for the standard representation 
of $GL(j),$ also realized as column vectors.  
Keeping in mind that $\jla$ inherits an action 
of $Sp(2n) \times GL(j)$ from inclusion into 
$\f{sp}(2n+2j),$ one has isomorphisms
$$
\f{sp}(2n) \cong \sym^2 V_{2n},
\qquad
v_1 \cdot v_2\in \sym^2 V_{2n} 
 \mapsto (v_1\;^\top \!v_2 +  v_2\;^\top \!v_1)J \in \f{sp}(2n), $$ 
$$\f v \cong U_j \otimes V_{2n},
\qquad
v\otimes u\in U_j \otimes V_{2n}
 \mapsto [u \;^\top \!v,0]\in \f v,$$
 $$ \f z \cong \sym^2 U_j, 
\qquad u_1\cdot u_2 \in \sym^2 U_j
\mapsto u_1\;^\top \!u_2 +  u_2\;^\top \!u_1.
$$
(Here, we identify the $j\times 2n$ matrix $u \;^\top \!v$
 with a pair of $j \times n$ matrices $\nu, \mu$ 
 to obtain an element of $\f v$ in the usual form.)

The vector space $V_{2n}$ is the direct sum
of two $n$-dimensional isotropic subspaces
$W, W^\perp$ such that
$\f l \leftrightarrow W \otimes U_j$ and 
$\f r \leftrightarrow W^\perp \otimes U_j.$
Specifically, $W$ is the span of the first $n$ standard 
basis vectors and $W^\perp$ is the span of he
last $n.$ 

Hence $v_1 \cdot v_2\in \sym^2 V_{2n}$ corresponds to a matrix 
of the form 
$$
\bpm X&0\\0& -^\top X\epm
$$
if and only if one of $v_1, v_2$ lies in $W$ and the 
other in $W^\perp.$
Without loss of generality one may assume $v_1 \in W$
and $v_2 \in W^\perp.$

Now, let $u_1, \dots, u_j$ denote the standard 
basis basis for $U_j.$  Then the mapping $T$ 
can be expressed in this notation as$$
T( v_1 \cdot v_2) 
= \sum_{i_1,i_2=1}^j 
(v_1\otimes u_{i_1})\wt{\b Z}_{i_1,i_2}
(v_2\otimes u_{i_2}).
$$
Furthermore,
$$[0,0,\zeta]\cdot v= 2\pi i \Tr( \c M \zeta )v
\implies \b Z_{ij} v = 4\pi i \c M_{ij} v.
$$
Finally, 
$$\begin{aligned}
\lambda\left((v_1\otimes u_{i_1})\wt{\b Z}_{i_1,i_2}
(v_2\otimes u_{i_2})\right)
= \frac 12\wt{\b Z}_{i_1,i_2}
\left((v_1\otimes u_{i_1})(v_2\otimes u_{i_2})
+(v_2\otimes u_{i_2})(v_1\otimes u_{i_1})
\right)\\
= \frac 12\wt{\b Z}_{i_1,i_2}
\left([(v_1\otimes u_{i_1}),(v_2\otimes u_{i_2})]
+2(v_2\otimes u_{i_2})(v_1\otimes u_{i_1})
\right).
\end{aligned}
$$
Now, tracing through the definitions, 
$[(v_1\otimes u_{i_1}),(v_2\otimes u_{i_2})]$
is precisely $\b Z_{i_1,i_2}.$
Hence
$$\sum_{i_1,i_2=1}^j 
\wt{\b Z}_{i_1,i_2}
[(v_1\otimes u_{i_1}),(v_2\otimes u_{i_2})]
= j(\det \b Z).
$$
On the other hand, $v_1 \in W \implies v_1 \otimes u_{i_1} 
\in \f l \implies (v_1 \otimes u_{i_1})\cdot v=0$
for all $1 \le i_1 \le j.$  
The result follows.
\end{proof}
\begin{defn}\label{def: b M+N and M+N}
Fix a positive integer $N.$  
Let $\hat{\b M}_{+,N}$ to be the $N\times N$ matrix 
whose $i,j$ entry is $E_{i,N+j}+E_{j,N+i} \in \f{sp}(2N) 
\subset U( \f{sp}(2n)).$  
Let $\hat M_{+,N}$ be the element of $U(\f{sp}(2N))$
obtained by taking the determinant of this matrix.  
\end{defn}
One may note that the entries of $\hat{\b M}_{+,N}$
all commute in $U(\f{sp}(2N)),$ so one does not 
need to be concerned about the orders in which the 
products are taken in taking the determinant.
Furthermore, $\hat{M}_{+,n+j}\in U(\jla)\subset U( \f{sp}(2n+2j)).$ 
\begin{lem}\label{lemma relating M+n and M+n+j}
One has
$(\det \b Z)^j\b T(\hat M_{+,n}) = (\det \b Z)^n\hat M_{+, n+j}.$
\end{lem}
\begin{proof}
Let $\b U$ be the $j \times n$ matrix with 
$i,k$ entry equal to 
$
[0,E_{i,k},0]\in U( \f h^{(n,j)}) \subset U( \f{sp}(2n+2j)).
$
Then 
$$
\hat{\b M}_{+,n+j}
= \bpm 
\hat{\b M}_{+,n} & ^\top  \b U\\
\b U& \b Z
\epm.
$$
The mapping $\b T:U( \f{sp}(2n))\to U(\jla)$
induces a mapping 
$\Mat_{n\times n} (U( \f{sp}(2n))) \to \Mat_{n\times n} (U( \jla))$
which we denote by the same symbol $\b T.$ 
Then 
$\b T(\hat{\b M}_{+,n}) 
= \det \b Z \hat{\b M}_{+,n} - \b U \wt{\b Z}{}\;^\top \b U.$
Hence 
$$\begin{aligned}
\det \b Z^n \cdot\hat M_{+, n+j}&= \det 
\bpm
\det \b Z\cdot I_n&-{}^\top \b U\wt{\b Z} \\
0&I_j
\epm
\bpm 
\hat{\b M}_{+,n} & ^\top  \b U\\
\b U& \b Z
\epm\\
&= \det\bpm 
\det \b Z\cdot \hat{\b M}_{+,n} - {}^\top  \b U\wt{\b Z}\b U&0\\
\b U& \det \b Z\cdot I_j
\epm\\
&=
(\det \b Z)^j\det \b T(\hat{\b M}_{+,n}),
\end{aligned}
$$
which gives the result.
\end{proof}

\begin{proof}[Proof of theorem \ref{Jacobi U(g) version of Conjecture 2}]
Theorem \ref{Jacobi U(g) version of Conjecture 2}
now follows from proposition \ref{U(g) version of Conjecture 2}.  Regard $V$ as a $U(\f{sp}(2n, \C))$
module with the action $*.$ 

It follows from lemma \ref{lem:  holomorphic for . implies holomorphic for *} that $v_0$ is holomorphic.
By \ref{lemma relating . and * for Siegel Levi} it
is of weight $-r+\frac{n+1}2.$
Therefore, by proposition \ref{U(g) version of Conjecture 2}, 
$\hat M_{+,n}^r*v_0$ is holomorphic
and of weight $r+\frac{n+1}2$
relative to the action $*.$
By lemma \ref{lemma relating M+n and M+n+j}, 
and the fact that $\c M$ is assumed positive 
definite, the same is true of $\hat M_{+,n+j}^r\cdot v_0.$
Lemma \ref{lemma relating . and * for Siegel Levi}
shows that $\hat M_{+,n+j}^r\cdot v_0$
is of weight $r+\frac{n+j+1}2$ relative to the action $\cdot.$
Now, 
$$
\bpm X&0\\0&-{}^\top\!\!X\epm 
 \hat M_{+, n+j} - \hat M_{+,n+j} \bpm X&0\\0&-{}^\top\!\! X\epm 
= \Tr X \hat M_{+, n+j}, 
\qquad (\forall X \in \f{gl}(n+j,\C)).
$$
It immediately follows that 
$$\bpm 
0&0&0&0\\\lambda&0&0&0\\ 0&0&0&-{}^\top\!\! \lambda\\
0&0&0&0
\epm\hat M_{+, n+j}^k \cdot v_0
=0, \qquad  \forall 
\lambda \in \Mat_{j\times n}\C,\;
k\in \N.$$
It follows from this, along with lemma \ref{lem:  holomorphic for . implies holomorphic for *} that 
$\hat M_{+, n+j}^k \cdot v_0$ is holomorphic
if and only if $k=r.$
\end{proof}

\end{document}